\documentclass[12pt]{article}
\usepackage[centertags]{amsmath}
\usepackage{amsfonts}
\usepackage{amssymb}
\usepackage{latexsym}
\usepackage{amsthm}
\usepackage{newlfont}
\usepackage{graphicx}
\usepackage{abstract}
\newlength{\defbaselineskip}
\newcommand{\setlinespacing}[1]%
           {\setlength{\baselineskip}{#1 \defbaselineskip}}

{\medskip\noindent \textit{Proof of #1. }}%
{\actaqed \medskip}

\newcommand{\IN}{\mathbb{N}}
\newcommand{\IB}{\mathbb{B}}
\newcommand{\IR}{\mathbb{R}}

\newtheorem{Theorem}{Theorem}[section]

\newtheorem{Proposition}{Proposition}[section]
\newtheorem{remark}{Remark}[section]

\numberwithin{equation}{section}

\begin{document}

\title{{On the orthogonal democratic systems  in  the $L^p$ spaces}\thanks{\it Math Subject Classifications.
primary:  41A65; secondary: 41A25, 41A46, 46B20.}}
\author{K. Kazarian \thanks{ Dept. of Mathematics, Mod. 17, Universidad Aut\'onoma de Madrid, 28049, Madrid, Spain e-mail: kazaros.kazarian@uam.es }
and A. San Antol\'{\i}n \thanks{Dept. of Mathematics, Universidad de Alicante, 03690 Alicante, Spain  e-mail: angel.sanantolin@ua.es}}

 \maketitle

\begin{abstract}
{The concept of bidemocratic pair for a Banach space was introduced in \cite{KS:18}.
 We construct a family of orthonormal systems $\mathfrak{F}_{l},$ $l\in (0,\infty)$ of functions defined on $[-1,1]$ such that
 the pair $(\mathfrak{F}_{l},\mathfrak{F}_{l})$ is bidemocratic for $L^{p}[-1,1]$  and for $L^{p'}[-1,1]$  if $l\in (0, \frac{p}{2(p-2)}]$, where $p>2$ and $p'= \frac{p}{p-1}$.    The system $\mathfrak{F}_{l}$  is not democratic for $L^{p'}[-1,1]$ when $l\in (\frac{p}{2(p-2)}, \frac{p}{p-2}). $
When $l> \frac{p}{2(p-2)}$ the pair $(\mathfrak{F}_{l},\mathfrak{F}_{l})$ is not bidemocratic neither for $L^{p}[-1,1]$ nor for $L^{p'}[-1,1]$.  }
\end{abstract}

\section{Introduction}

Greedy algorithms have been studied extensively during last two decades. S.V. Konyagin and V.N. Temlyakov \cite{KonT}  gave a
characterization  of greedy bases: a basis is greedy if and only if it is unconditional and democratic.
An infinite system $X =\{x_{k}\}_{k=1}^{\infty}$ in a Banach space $\IB$ will be called
  a democratic  system for $\IB$ if there exists a constant $D >1$ such that, for any two finite sets of indices $P$ and $Q$ with the same cardinality $|P|=|Q|$, we have
\begin{equation}\label{dem:1}
\bigg\|\sum_{k\in P} \frac{x_k}{\|x_k\|}\bigg\| \le D\bigg\|\sum_{k\in Q} \frac{x_k}{\|x_k\|}\bigg\|.
\end{equation}
 A pair of systems $X =\{x_{k}\}_{k=1}^{\infty} \subset \IB, $ $X^{*} =\{x^{*}_{k}\}_{k=1}^{\infty} \subset \IB^{*} $ is called biorthogonal if $ x^{*}_{k}(x_{m}) = \delta_{km}$, where $\delta_{km}$ is the Kronecker symbol. In \cite{DKKT} bidemocratic bases have been studied.
Following \cite{DKKT} we put
\[
\varphi_{X}(n) = \sup_{|P|\leq n} \bigg\|\sum_{k\in P} \frac{x_k}{\|x_k\|} \bigg\|_{\IB}, \qquad \varphi_{X^{*}}^{*}(n) = \sup_{|P|\leq n} \bigg\|\sum_{k\in P} \|x_k\| \, x^{*}_k \bigg\|_{\IB^{*}}
\]
and will say that a pair of biorthogonal systems $(X,X^{*})$ is bidemocratic for $\IB$ if there exists $C>0$ such that for any  $n\in \IN$
\begin{equation}\label{bide:1}
\varphi_{X}(n) \varphi_{X^{*}}^{*}(n)  \leq C n.
\end{equation}
Modifying the definition given in \cite{DKKT} we say that $\varphi_{X}(n)$ is the fundamental function and $\varphi_{X^{*}}^{*}(n)$ is the dual fundamental function of the pair of biorthogonal systems $(X,X^{*})$. It is proved in \cite{DKKT} that a bidemocratic basis is a democratic basis. The above definition of bidemocratic system is given for minimal systems which are not necessarily bases.
Further we will check that if a pair of biorthogonal systems $(X,X^{*})$ is bidemocratic for $\IB$ then the system $X$ is democratic in $\IB$.
 It is clear that if a system is democratic for $\IB$ then any its infinite subsystem is also democratic. Using the concept of bidemocratic pair we find conditions for which the inverse assertion is also true. This idea was used in \cite{KS:18} (see also \cite{KKS:18}) to give a complete characterization of weight functions $\omega$ for which the higher rank Haar wavelets are bidemocratic systems for $L^{p}(\IR,\omega), 1<p<\infty$.

One of the main purposes of the article \cite{DKKT} was the study of the duality properties of the greedy algorithms. For example, if $\IB$ is a reflexive Banach space, the pair of biorthogonal systems $(X,X^{*})$ is bidemocratic for $\IB$ and $\|x_{j}\|_{\IB}\cdot \|x^{*}_{j}\|_{\IB^{*}} = \theta, j\in \IN$ for some $\theta \geq 1$ then the pair of biorthogonal systems $(X^{*}, X)$ is bidemocratic for $\IB^{*}$. Of course, we came to the same conclusion if $\varphi_{X} \asymp \varphi^{*}_{X^{*}}$ and $\varphi^{*}_{X} \asymp \varphi_{X^{*}}.$
 We say that $\varphi$ and $\psi$ are equivalent, $\varphi \asymp \psi$ if $\varphi$ and $\psi$ defined on $\IN$ with values in $\IR^{+} = \{t\in \IR: t \geq 0\}$ and for some $0<C_{1}<C_{2}$ we have that $C_{1} \varphi(n) \leq \psi(n) \leq C_{2}\varphi(n), n\in \IN$.

 We construct a family of orthonormal systems such that they are bidemocratic for $L^{p}$ but for a subset of parameters they are not democratic for the dual space $L^{p'}$, for another set of parameters those systems are democratic for $L^{p'}$ but not bidemocratic for $L^{p}.$ Finally, for another set of parameters they are bidemocratic for $L^{p'}$.

The characteristic function of a set $E$ is denoted by $I_E$ and $\IN_{0} = \IN\bigcup\{0\}$. Let $E \subseteq\IR, |E|>0$ be a  measurable set  then we write $\phi \in L^{p}(E), 1 \leq p <\infty$
if $\phi:E \rightarrow \IR$ is measurable on $E$ and the norm is defined by
\[
 \| \phi \|_{L^{p}(E)}:= \left( \int_{E} |\phi(t)|^{p}  dt\right)^{\frac{1}{p}} < +\infty.
 \]

\

\section{Democratic systems}

Let $\IN_{j} \subset \IN, 1\leq j \leq \nu$ be such that $\IN_{j} \bigcap \IN_{i} = \emptyset$ if $i\neq j,$ $\text{card}\, \IN_{j} = \infty, 1\leq j \leq \nu$ and $\bigcup_{j=1}^{\nu} \IN_{j} = \IN$. For a given pair of biorthogonal systems $(X,X^{*})$ consider the biorthogonal pairs $(X_{j},X_{j}^{*}),$ where
$X_{j} = \{ x_{k} \}_{k\in \IN_{j}},$ $ X^{*}_{j} = \{ x^{*}_{k} \}_{k\in \IN_{j}},$ $ 1\leq j \leq \nu.$

\begin{Proposition}\label{prop:2.3}
Let $(X_{j},X_{j}^{*}), 1\leq j \leq \nu$ be  pairs of biorthogonal systems defined as above. If pairs $(X_{j},X_{j}^{*}), 1\leq j \leq \nu$ are bidemocratic for $\IB$ and for any $1\leq i<j\leq \nu$ the functions  $\varphi_{X_{i}}(\cdot), \varphi_{X_{j}}(\cdot)$ are equivalent then
  the pair of  biorthogonal systems $(X, X^{*})$ is bidemocratic for $\IB$. Moreover,    $\varphi_{X}(\cdot)$ and $\varphi_{X_{j}}(\cdot)$ are equivalent for any $1\leq j \leq \nu$.
\end{Proposition}
\begin{proof}
Let
\[
\tilde{\varphi}_{X}(n) = \max_{1\leq j\leq \nu} \varphi_{X_{j}}(n) \qquad n\in \IN.
\]
We have that for any $n\in \IN$
\[
 \frac{1}{\nu} \varphi_{X}(n) \leq   \tilde{\varphi}_{X}(n) \leq  \varphi_{X}(n).
\]
The right hand inequality is obvious. On the other hand
\[
 \varphi_{X}(n) \leq \sum_{j=1}^{\nu}  \varphi_{X_{j}}(n)  \leq \nu \, \tilde{\varphi}_{X}(n).
 \]
Let $P\subset \IN$ be a finite set. We have that
\begin{align*}\label{dem:1}
  &\varphi_{X}(|P|) \varphi_{X^{*}}^{*}(|P|)  \leq \nu \, \tilde{\varphi}_{X}(|P|) \sum_{j=1}^{\nu} \varphi_{X_{j}^{*}}^{*}(|P|)\\
&\leq
\nu \,  \sum_{j=1}^{\nu} C_{j} \varphi_{X_{j}}(|P|) \varphi_{X_{j}^{*}}^{*}(|P|) \leq \nu \,C \sum_{j=1}^{\nu} C'_{j}|P|,
\end{align*}
where $C= \max_{1\leq j\leq \nu} C_{j}$.

\end{proof}

The condition \eqref{bide:1} yields
\begin{remark}\label{rem:1}
If a pair of biorthogonal systems $(X,X^{*})$ is bidemocratic for $\IB$ then there exists $C>0$  such that for any $k\in \IN$
\begin{equation}\label{bide:2}
\|x_k\|_{\IB}\cdot \|x^{*}_k\|_{\IB^{*}} \leq C .
\end{equation}
\end{remark}

It is proved in \cite{DKKT} that a bidemocratic basis is a democratic basis.
The proof given in \cite{DKKT} for bases also works for the proof of the following
\begin{Proposition}\label{prop:2.1}
Let $(X,X^{*})$ be a pair of biorthogonal systems   bidemocratic for $\IB$. Then   the system $X$ is democratic for $\IB$.
\end{Proposition}

We are going to construct a family of orthonormal systems in order to clarify some duality properties of orthonormal systems if it is democratic  for the $L^{p}, 1<p<\infty$ spaces.

Let $\chi$ be an orthonormal system of functions defined on $[-1,1]$ as follows:

For any $n\in \IN$ we divide the interval $(-2^{-n+1}, -2^{-n}]$ into $2^{n}$ equal intervals $\Delta_{j}^{n}, 1\leq j\leq 2^{n}$ such that $\Delta_{j}^{n}\bigcap \Delta_{i}^{n} = \emptyset$ if $i\neq j$.

Set $\chi_{j}^{n}(x) = 2^{n}I_{\Delta_{j}^{n}}(x): 1\leq j\leq 2^{n}, n\in \IN$, where $I_{E}(\cdot)$ is the characteristic function of the set $E\subset [-1,1]$.
It is clear that the system
$\chi = \{ \chi_{j}^{n}(x) : 1\leq j\leq 2^{n}, n\in \IN \},$  is an orthonormal system of functions on $[-1,1]$.

 We put $k_{0}=0$ and $k_{n} = k_{n-1} + 2^{n} = 2(2^{n}-1), n\in \IN$. In our construction we use the Rademacher system $\{r_{k}(t)\}_{k=1}^{\infty},$ which is an orthonormal system of functions defined on $[0,1]$(see \cite{KaSt:35},\cite{Kah:93}).
Let
\[
f_{j}^{(n,l)}(x) =
\begin{cases}
\sqrt{1-2^{-\frac{n}{l}}} \chi_{j}^{n}(x),     &\text{if $x\in [-1,0)$;}\\
2^{-\frac{n}{2l}}r_{k_{n-1}+j}(x)     &\text{if $x\in [0,1]$,}
\end{cases}
\]
$1\leq j\leq 2^{n}, n\in \IN $ and $l\in (0, \infty)$.

 Let $f^{l}_{k}(x) = f_{j}^{(n,l)}(x)$ if $k= k_{n-1}+j$ and $1\leq j\leq 2^{n}$.
For any fixed $l\in (0, \infty)$ the system
$
\mathfrak{F}_{l} = \{ f_{k}^{l}(x) \}_{k=1}^{\infty}
$
is an orthonormal system of functions defined on $[-1,1]$.

\begin{Proposition}\label{prop:2.4}
For any $l\in (0, \infty)$
 he system $\mathfrak{F}_{l}$ is a democratic system for $L^{p}{[-1,1]},$  $ 2\leq p <\infty,$  and  $\varphi_{\mathfrak{F}_{l}}(m) \asymp m^{\frac{1}{p}}$.
\end{Proposition}
\begin{proof}
The proposition is obviously true if $p=2$. Thus we only will consider the case $p>2$.
We have that
\begin{equation}\label{eq:1.0}
|c_{n,l}|^{p}:= \bigg\|  f_{j}^{(n,l)} \bigg\|^{p}_{L^{p}{[-1,1]}} = 2^{n(p-2)}(1-2^{-\frac{n}{l}})^{\frac{p}{2}} +2^{-\frac{np}{2l}}.
\end{equation}
Let $|b_{k,l}|^{p}:=|c_{n,l}|^{p}$ if $k= k_{n-1}+j$ and $1\leq j\leq 2^{n}$.

We prove that there exists $0 <C_{1}(l,p) < C_{2}(l,p) $ such that for any finite set $A\subset \IN, |A| =m$
\begin{equation}\label{eq:1a}
C_{1}(l,p) m^{\frac{1}{p}} \leq \bigg\| \sum_{k\in A} \frac{1}{|b_{k,l}|} f_{k}^{l} \bigg\|_{L^{p}{[-1,1]}} \leq C_{2}(l,p) m^{\frac{1}{p}}.
\end{equation}
Let $n_{0}=  [l]+1,$ where $ [l] \in \IN$ and $[l]\leq l <[l]+1$. Thus $2^{-\frac{n p}{2l}} \leq 2^{-\frac{p}{2}}$ if $n\geq n_{0}$. We have that
\[
2^{-\frac{np}{2l}}\cdot 2^{-n(p-2)}(1-2^{-\frac{n}{l}})^{-\frac{p}{2}} \leq 2^{-\frac{p}{2}}2^{2-p} 2^{\frac{p}{2}} = 2^{2-p}
\]
when $n\geq n_{0}$.
Then for any $A\subset [n_{0},\infty) \bigcap \IN$, $|A|= m$
we have that
\[
  \bigg\| \sum_{k\in A}\frac{1}{|b_{k,l}|} f_{k}^{l} \bigg\|^{p}_{L^{p}{[-1,0]}} \geq  \frac{m}{1+ 2^{2-p} } > \frac{m}{2}.
\]
The supports of functions on $[-1,0)$ are not empty and do not coincide. Thus changing the constant we easily get the left side inequality in \eqref{eq:1a} for the general case.

Let $m> 2^{n_{0}}, m\in \IN$ and $\nu \in \IN$ be such that $2^{\nu-1} <m \leq 2^{\nu}$.
By the Khintchine inequality (see \cite{KaSt:35}) it follows that
\[
\int_{[0,1]}\bigg|\sum_{k\in A}\frac{1}{|b_{k,l}|} f_{k}^{l}(x)\bigg|^{p} dx \leq D_{p} \bigg(\sum_{n=1}^{\nu} |c_{n,l}|^{-2} 2^{-\frac{n}{l}} 2^{n}\bigg)^{\frac{p}{2}}.
\]
Clearly
\begin{equation}\label{eq:p11}
|c_{n,l}|^{p} \geq \frac{1}{2} 2^{n(p-2)}\qquad \text{for $n\geq n_{0}$} .
\end{equation}
Thus it follows that
\[
\sum_{n=n_{0}}^{\nu} |c_{n,l}|^{-2} 2^{-\frac{n}{l}} 2^{n} \leq 2^{\frac{2}{p}} \sum_{n=n_{0}}^{\nu} 2^{-\frac{2n(p-2)}{p}}2^{n(1-\frac{1}{l})}.
\]
Let $\kappa = \frac{4}{p}-1-\frac{1}{l}$. If $\kappa >0$ we write
\[
\sum_{n=n_{0}}^{\nu} |c_{n,l}|^{-2} 2^{-\frac{n}{l}} 2^{n} \leq 2^{\frac{2}{p}} \sum_{n=0}^{\nu} 2^{\kappa n} \leq 2^{\frac{2}{p}} \frac{1}{2^{\kappa}-1} 2^{\kappa(\nu +1)} \leq  \frac{2^{2(\kappa + \frac{1}{p})}}{2^{\kappa}-1} m^{\kappa} = C_{p}m^{\kappa} .
\]
Thus it follows
\begin{align*}
  &\bigg\| \sum_{k\in A}\frac{1}{|b_{k,l}|} f_{k}^{l} \bigg\|^{p}_{L^{p}{[-1,1]}} \leq  m + D_{p} (C'_{p} + C_{p}m^{\kappa})^{\frac{p}{2}}\\ &=m(1+ D_{p} (C'_{p}m^{-\frac{2}{p}}+ C_{p}m^{\kappa -\frac{2}{p}})^{\frac{p}{2}}).
\end{align*}
Whence we  obtain the right hand inequality in \eqref{eq:1a} because $\kappa -\frac{2}{p}<0$. If $\kappa\leq 0$ then the proof is obvious.
\end{proof}

\begin{Proposition}\label{prop:2.5}
 The system $\mathfrak{F}_{l}$ is a democratic system for $L^{r}{[-1,1]},$  $ 1\leq r <2,$  $l\in (0, \frac{r}{2(2-r)}]$  and  $\varphi_{\mathfrak{F}_{l}}(m) \asymp m^{\frac{1}{r}}$.
\end{Proposition}
\begin{proof}
We have that
\begin{equation}\label{eq:r1}
|\hat{c}_{n,l}|^{r}:= \bigg\|  f_{j}^{(n,l)} \bigg\|^{r}_{L^{r}{[-1,1]}} = 2^{n(r-2)}(1-2^{-\frac{n}{l}})^{\frac{r}{2}} +2^{-\frac{nr}{2l}}.
\end{equation}
As above we put  $|\hat{b}_{k,l}|^{r}:=|\hat{c}_{n,l}|^{r}$ if $k= k_{n-1}+j$ and $1\leq j\leq 2^{n}$.

If $l\in (0,\frac{r}{2(2-r)}]$ then $2^{n(r-2)} \geq 2^{-\frac{nr}{2l}}$ and it follows that
\[
  \bigg\| \sum_{k\in A}\frac{1}{|\hat{b}_{k,l}|} f_{k}^{l} \bigg\|^{r}_{L^{r}{[-1,1]}} \geq  \frac{m}{2}(1-2^{-\frac{1}{l}})^{\frac{r}{2}}.
\]
On the other hand we have that
there exists $n_{1}\in \IN$ such that
\[
|\hat{c}_{n,l}|^{r} \geq \frac{1}{2} 2^{n(r-2)}\qquad \text{for $n\geq n_{1}$} .
\]
Let $A_{n} = [k_{n-1}+1, k_{n}]\bigcap A, n\in \IN$ and $\Omega_{A}=\{n\in \IN: A_{n}\neq \emptyset \}$.

Let $m> 2^{n_{1}}, m\in \IN$ and $\nu \in \IN$ be such that $2^{\nu-1} < m \leq 2^{\nu}$ then by the Khintchine inequality it follows that
\[
\int_{[0,1]}\bigg|\sum_{k\in A}\frac{1}{|\hat{b}_{k,l}|} f_{k}^{l}(x)\bigg|^{r} dx \leq D_{r} \bigg(\sum_{n\in \Omega_{A} } |\hat{c}_{n,l}|^{-2} 2^{-\frac{n}{l}} |A_{n}| \bigg)^{\frac{r}{2}}
\]
\[
\leq 2 D_{r}  \bigg(\sum_{n\in \Omega_{A} } 2^{\frac{2n(2-r)}{r}}  2^{-\frac{n}{l}} |A_{n}| \bigg)^{\frac{r}{2}} \leq 2 D_{r} m^{\frac{r}{2}}
\]
if $l\leq \frac{r}{2(2-r)}.  $
 Afterwards
we proceed as in the proof of Proposition \ref{prop:2.4} and easily finish the proof.

\end{proof}

\begin{Proposition}\label{prop:2.6}
 The system $\mathfrak{F}_{l}$ is not a democratic system for $L^{r}{[-1,1]},$ $ 1\leq r <2$ if $l\in ( \frac{r}{2(2-r)}, \frac{r}{2-r})$.
\end{Proposition}
\begin{proof}
If $l> \frac{r}{2(2-r)}$ then $2^{n(r-2)} < 2^{-\frac{nr}{2l}}$. Thus
 for any $n\in \IN$
\begin{equation}\label{eq:261}
2^{-\frac{nr}{2l}} \leq |\hat{c}_{n,l}|^{r} \leq 2\cdot 2^{-\frac{nr}{2l}},
\end{equation}
where $\hat{c}_{n,l}$ is defined by \eqref{eq:r1}.
Let $B_{n} = [k_{n-1}+1, k_{n}]\bigcap \IN, n\in \IN$. Then it follows that
\begin{align*}\label{dem:2}
 &\bigg\| \sum_{k\in B_{n}}\frac{1}{|\hat{b}_{k,l}|} f_{k}^{l} \bigg\|^{r}_{L^{r}_{[-1,1]}} =  \bigg\|\frac{1}{|\hat{c}_{n,l}|} \sum_{k\in B_{n}} f_{k}^{l} \bigg\|^{r}_{L^{r}_{[-1,1]}} \geq \frac{1}{2} 2^{\frac{nr}{2l}} 2^{n(r-1)}(1-2^{-\frac{n}{l}})^{\frac{r}{2}} \\
 & + \frac{1}{2} \bigg\| \sum_{j= k_{n-1}+1}^{k_{n}} r_{j}(\cdot) \bigg\|^{r}_{L^{r}_{[0,1]}} \geq \frac{1}{2} 2^{n \omega} (1-2^{-\frac{n}{l}})^{\frac{r}{2}} +D^{*}_{r} 2^{\frac{nr}{2}},
\end{align*}
where $\omega = \frac{r}{2l} +r-1$ and $D^{*}_{r} >0$. Observe that
\[
\frac{r}{2} < \omega <1 \quad \text{if $l\in \bigg( \frac{r}{2(2-r)}, \frac{r}{2-r}\bigg)$.}
\]
Afterwards we consider $B^{*}_{n} = [k_{n^{2}-1}+1, k_{n^{2}-1}+2^{n}]\bigcap \IN, n\in \IN$. In this case we have
\begin{align*}\label{dem:3}
 &\bigg\| \sum_{k\in B^{*}_{n}}\frac{1}{|\hat{b}_{k,l}|} f_{k}^{l} \bigg\|^{r}_{L^{r}_{[-1,1]}} =  \bigg\|\frac{1}{|\hat{c}_{n^{2},l}|} \sum_{k\in B^{*}_{n}} f_{k}^{l} \bigg\|^{r}_{L^{r}_{[-1,1]}} \leq  2^{\frac{n^{2}r}{2l}} 2^{n^{2}(r-2)}2^{n}(1-2^{-\frac{n^{2}}{l}})^{\frac{r}{2}} \\
 & + \bigg\| \sum_{j= k_{n^{2}-1}+1}^{k_{n^{2}-1}+2^{n}} r_{j}(\cdot) \bigg\|^{r}_{L^{r}_{[0,1]}} \leq \frac{1}{2} 2^{n^{2} \omega_{1}+n} (1-2^{-\frac{n^{2}}{l}})^{\frac{r}{2}} +D_{r} 2^{\frac{nr}{2}},
\end{align*}
where $D_{r} >0$ and $\omega_{1} = \frac{r}{2l} +r-2 = \omega - 1 < 0 $ if $l\in \bigg( \frac{r}{2(2-r)}, \frac{r}{2-r}\bigg)$. 
The inequality $\frac{r}{2} < \omega$ yields that the system $\mathfrak{F}_{l}$ is not a democratic system.
\end{proof}

\begin{Proposition}\label{prop:2.7}
 The system $\mathfrak{F}_{l}$ is a democratic system for $L^{r}{[-1,1]},$ $ 1\leq r <2,$ if $l\in [ \frac{r}{2-r}, \infty)$.  Moreover,  $\varphi_{\mathfrak{F}_{l}}(n) \asymp n^{\frac{1}{2}}$.
\end{Proposition}
\begin{proof}
If $l\in [ \frac{r}{2-r}, \infty)$ we have that the inequalities \eqref{eq:261} hold. Hence, for any $A\subset \IN, |A|=m$
\[
 \bigg\| \sum_{k\in A}\frac{1}{|\hat{b}_{k,l}|} f_{k}^{l} \bigg\|^{r}_{L^{r}{[-1,1]}} \geq  \frac{1}{2} \bigg\| \sum_{n\in \Omega_{A} } \sum_{k\in A_{n}} r_{k} \bigg\|^{r}_{L^{r}{[0,1]}} \geq C_{r} m^{\frac{r}{2}},
 \]
 where the last inequality follows by the Khintchine inequalities and $C_{r} >0$.
Let $\nu \in \IN$ be such that $2^{\nu-1} \leq  m < 2^{\nu}$. Then
\begin{align*}
  &\bigg\| \sum_{k\in A}\frac{1}{|\hat{b}_{k,l}|} f_{k}^{l} \bigg\|^{r}_{L^{r}{[-1,0]}} \leq  \sum_{n\in \Omega_{A} } \sum_{k\in A_{n}} 2^{n(r-2+\frac{r}{2l})}\\
   &\leq \sum_{n\in \Omega_{A} } 2^{n(\frac{r}{2}-1)} |A_{n}| \leq \sum_{k=1}^{\nu} 2^{k\frac{r}{2}} \leq \frac{2^{\frac{r}{2}} }{2^{\frac{r}{2}} -1 } m^{\frac{r}{2}}.
\end{align*}
By Khintchine's inequalities we obtain 
  \[
  \bigg\| \sum_{k\in A}\frac{1}{|\hat{b}_{k,l}|} f_{k}^{l} \bigg\|^{r}_{L^{r}{[0,1]}} \leq  \bigg\| \sum_{n\in \Omega_{A} } \sum_{k\in A_{n}} r_{k} \bigg\|^{r}_{L^{r}{[0,1]}} \leq D_{r} m^{\frac{r}{2}}.
\]

\end{proof}

Resuming the propositions proved above we easily obtain the following theorem.

\begin{Theorem}\label{T:0.1} Let $p>2$ and $p'= \frac{p}{p-1}$. Then the pair $(\mathfrak{F}_{l},\mathfrak{F}_{l})$ is bidemocratic for $L^{p}[-1,1]$  and for $L^{p'}[-1,1]$  if $l\in (0, \frac{p}{2(p-2)}].$    The system $\mathfrak{F}_{l}$  is not democratic for $L^{p'}[-1,1]$ when $l\in (\frac{p}{2(p-2)}, \frac{p}{p-2}). $

When $l> \frac{p}{2(p-2)}$ the pair $(\mathfrak{F}_{l},\mathfrak{F}_{l})$ is not bidemocratic neither for $L^{p}[-1,1]$ nor for $L^{p'}[-1,1]$.
\end{Theorem}

\

\end{document}